\documentclass[12pt]{amsart}
\usepackage{amsmath, amssymb, amsthm, verbatim, hyperref, graphicx}
\usepackage[all, cmtip]{xy} 

\begin{document}

\title{The lower central series of the symplectic quotient of a free associative algebra}
\author{Ben Bond, David Jordan}
\address{Department of Mathematics, MIT, Cambridge, MA 02139\\
Department of Mathematics, The University of Texas at Austin, Austin TX 78712}

\begin{abstract}
We study the lower central series filtration $L_k$ for a symplectic quotient $A=A_{2n}/\langle\omega\rangle$ of the free algebra $A_{2n}$ on $2n$ generators, where $\omega=\sum [x_i,x_{i+n}]$.  We construct an action of the Lie algebra $H_{2n}$ of Hamiltonian vector fields on the associated graded components of the filtration, and use this action to give a complete description of the reduced first component $\bar{B}_1(A)= A/(L_2 + AL_3)$ and the second component $B_2=L_2/L_3$, and we conjecture a description for the third component $B_3=L_3/L_4$.

\end{abstract}

\maketitle
%\linenumbers

\section{Introduction and Results}

\newcounter{counter}
\numberwithin{counter}{section}

\newtheorem{thm}[counter]{Theorem}
\newtheorem{mydef}[counter]{Definition}
\newtheorem{lma}[counter]{Lemma}
\newtheorem{cor}[counter]{Corollary}
\newtheorem{prop}[counter]{Proposition}
\newtheorem{conj}[counter]{Conjecture}

\theoremstyle{remark}
\newtheorem{rem}[counter]{Remark}

\newcommand{\CC}{\mathbb{C}}
\newcommand{\ot}{\otimes}
\newcommand{\edit}[1]{\marginpar{#1}}
\newcommand{\spn}{\mathfrak{sp}_{2n}}
\newcommand{\x}{\underline{x}}
\newcommand{\y}{\underline{y}}
\newcommand{\z}{\underline{z}}

The lower central series of an associative algebra $A$ is the descending filtration by Lie ideals, $L_1(A):=A$, and $L_k(A):= [A,L_{k-1}(A)]$.  We denote by $M_k$ the two-sided associative ideal generated by $L_k$, and by $B_k(A)$ and $N_k(A)$ the associated graded components $B_{k}(A):=L_k(A)/L_{k+1}(A)$, and $N_k(A):=M_k(A)/M_{k+1}(A)$.  We also denote by $\bar{B}_1$ the quotient $\bar{B}_1:= A/(M_3+L_2)$.  The study of the components $B_k$ was initiated in \cite{FS}, and continued in a series of papers \cite{DE}, \cite{AJ}, \cite{BJ}.

Let $A_{m}$ denote the free algebra with generators $x_1,\ldots,  x_{m}$.  In the algebra $A_{2n}$, we define: $$\omega:=\frac{1}{2}\sum_{i=1}^n [x_i,x_{i+n}],$$
and denote by $\langle \omega \rangle$ the two-sided associative ideal generated by $\omega$.
\begin{mydef} The \emph{symplectic quotient} of the free algebra is: $$A'_{2n}:=A_{2n}/\langle \omega \rangle.$$ 
\end{mydef}

In \cite{FS}, an isomorphism of associative algebras was constructed between $A_m/M_3$ and the algebra $\Omega^{even}_{*}(\mathbb{C}^m)$ of even-degree differential forms, with Fedosov product $a\ast b:= ab + da\wedge db$.  This isomorphism maps $\omega\in A_{2n}$ to the standard symplectic form on $\CC^{2n}$ which, by abuse of notation, we also denote $\omega$.

The study of the components $B_k(A_m)$ and $N_k(A_m)$ has relied heavily upon an action of the Lie algebra, $W_m$, of polynomial vector fields on $\mathbb{C}^m$.  It is shown in \cite{DE} that each $B_k$ has a finite-length Jordan-H\"older series with respect to this action, whose composition factors are so-called ``tensor field modules".  In \cite{AJ}, and \cite{BJ}, bounds are given on the degree of these modules, which allow the components to be computed explicitly in many examples.

In the present paper, we construct an action of the Lie algebra $H_{2n}$ of Hamiltonian vector fields (i.e. vector fields which fix the form $\omega$) on the components $B_k(A'_{2n})$ and $N_k(A'_{2n})$.  By studying this action, we are able to generalize many of the results about $A_m$ to the symplectic quotients $A'_{2n}$.

In particular, the general framework discussed in Section \ref{action} yields isomorphisms,
$$A'_{2n}/M_3(A'_{2n}) \cong \Omega^{even}_{*}(\mathbb{C}^{2n})/\langle \omega \rangle,$$
$$\bar{B}_1(A_{2n}') \cong \Omega^{even}_{*}(\mathbb{C}^{2n})/( \Omega^{even,+}_{closed}(\mathbb{C}^{2n}) + \langle \omega \rangle),$$ 
$$ B_2(A'_{2n}) \cong \Omega^{even,+}_{closed}(\mathbb{C}^{2n})/ (\Omega^{even,+}_{closed}(\mathbb{C}^{2n}) \cap \langle \omega \rangle).$$

The irreducible representations of $H_{2n}$ which appear in the present work are certain tensor field modules $\mathcal{F}_\lambda$ associated to Young diagrams $\lambda \neq (1^k)$, and the irreducible sub-quotients of $\mathcal{F}_{(1^k)}$ (see Section \ref{f lambda def} for details). In fact, we show
\begin{prop}
As $H_{2n}$-modules, each $B_k(A'_{2n})$ and $N_k(A'_{2n})$, for $k\geq 2$, has a finite length Jordan H\"older series, consisting of tensor field modules $\mathcal{F}_\lambda$, with $\lambda \neq (1^k)$, and of subquotients $\mathcal{F}_{(1^k)}/T_k$, $Y_k/X_k$, $Z_k/X_k$, $X_k$ of $\mathcal{F}_{(1^k)}$.
\end{prop}
Our main results are a computation of the Jordan-H\"older series for the modules $A'_{2n}/M_3$, $\bar{B}_1(A'_{2n})$, $B_2(A'_{2n})$, and conjecturally for $B_3(A_{2n}')$.  We have: 
\begin{thm}\label{AmodM3thm}
The $H_{2n}$-module composition factors of $A'_{2n}/M_3(A'_{2n})$ are: 
\begin{align*}
\mathcal{F}_{(1^k)}/T_k,\, Y_k/X_k,\, Z_k/X_k,\, X_k   \hspace{.5 in} & \textrm{for k even, } 2\le k \le n-1,\\
%Y_k/X_k\hspace{.5 in} & 2\le k \le n-1, \mbox{  k even}\\
%Z_k/X_k\hspace{.5 in} & 2\le k \le n-1, \mbox{  k even}\\
%X_k\hspace{.5 in} & 2\le k \le n-1, \mbox{  k even}\\
\mathcal{F}_{(1^n)}/X_n,\, X_n \hspace{.5 in} &\textrm{if $n$ even}.\\
%X_n \hspace{.5 in} &\textrm{if $n$ even}\\
\mathcal{F}_{0}/X_0,\, X_0.\hspace{.5 in} &%\\
%X_0\hspace{.5 in} &
\end{align*}
\end{thm}
\begin{thm}\label{barB1}
The $H_{2n}$-module composition factors of $\bar{B}_1(A'_{2n})$ are: 
\begin{align*} 
\mathcal{F}_{(1^k)}/T_k,\, Y_k/X_k\hspace{.5 in} & \textrm{for $k$ even, } 2\le k \le n-1,\\
%Y_k/X_k\hspace{.5 in} & 2\le k \le n-1, \mbox{  k even}\\
\mathcal{F}_{(1^n)}/X_n \hspace{.5 in} &\textrm{if $n$ even}.\\
\mathcal{F}_{0}/X_0,\, X_0.\hspace{.5 in} &%\\
%X_0\hspace{.5 in}
\end{align*}
\end{thm}
\begin{thm}\label{B2}
The $H_{2n}$-module composition factors of $B_2(A'_{2n})$ are:
\begin{align*}
Z_k/X_k,\,X_k \hspace{.5 in}& \textrm{for $k$ even, } 2\le k \le n-1,\\ 
%X_k \hspace{.5 in}& 2\le k \le n-1, \mbox{  k even}\\
X_n \hspace{.5 in} &\textrm{if $n$ even}.
\end{align*}
\end{thm}

\begin{conj}\label{B3thm}
The $H_{2n}$-module composition factors of $B_3(A'_{2n})$ are:
\begin{align*}
\mathcal{F}_{(2,1^k)},\,\mathcal{F}_{(1^k)}/T_k,\,Z_k/X_k \hspace{.5 in}& \textrm{for $k$ odd, } 1\le k\le n-1.%\\
%\mathcal{F}_{(1^k)}/T_k \hspace{.5 in} & 1\le k\le n-1, \mbox{  k odd} \\
%Z_k/X_k\hspace{.5 in} & 1\le k\le n-1, \mbox{  k odd} \\
\end{align*}
\end{conj}

\begin{rem} In Section \ref{B3sec}, we show that the above list is an upper bound for the Jordan-H\"older series of $B_3(A'_{2n})$, so that the only ambiguity is whether the summands do in fact appear.  For $2n=4,6$, the conjecture is a theorem, based on MAGMA computations showing that certain cyclic generators of each term in the series are non-zero.
\end{rem}

The outline of this paper is as follows: In Section \ref{prelims}, we recall facts from the representation theory of the Lie algebras $W_n$ and $H_{2n}$ which we will need.  In Section \ref{action}, we construct an action of the Lie algebra $H_{2n}$ on each quotient $B_k(A'_{2n})$, and $\bar{B}_1(A'_{2n})$, and show that these are finite extensions of tensor field modules.  In Section \ref{Bkdescs}, we prove Theorems \ref{AmodM3thm}, \ref{barB1}, and \ref{B2}.  In Section \ref{B3sec} we describe $B_3(A_{2n})$ as a $H_{2n}$-module, and conjecture a description of $B_3(A'_{2n})$.

\subsection{Acknowledgments}
The authors would like to thank Pavel Etingof and Xiaoguang Ma for many helpful conversations during the course of this project; in particular the contents of Section 2.3 were explained to us by Pavel Etingof, who also conjectured Proposition 3.4.  Crucial evidence was collected using the Magma computational algebra system \cite{BCP}.  Finally, we are grateful to a careful referee for providing many helpful comments, and corrections to formulas in Proposition \ref{partialprop} and \ref{dcvprop}.

\section{Preliminaries} \label{prelims}
In this section we recall the Lie algebras $W_m$, $H_{2n}$, and $\mathfrak{sp}_{2n}$ of polynomial vector fields on $\CC^m$, Hamiltonian vector fields on $\CC^{2n}$, and Hamiltonian linear transformations of $\mathbb{C}^{2n}$, respectively.
\subsection{The symplectic Lie algebra and the restriction functor}
Let $E_{ij}\in\mathfrak{gl}_{2n}$ be the matrix with $1$ in the $i$-th row and $j$-th column and
$0$ everywhere else. Let $L_i$ be the dual basis to the diagonal span $\{E_{ii}\}_i$.
The Cartan subalgebra of $\mathfrak{sp}_{2n}$ is generated by:
$$H_i:=E_{i,i}-E_{i+n,i+n} , \mbox{ for } 1\le i\le n.$$
The positive roots are 
$$ \{L_i-L_j\}_{1\le i<j\le n}\,\, \bigcup \,\,\{L_i+L_j\}_{1\le i\le j \le n}.$$
The positive root vectors, with corresponding roots, are:
$$X_{ij}:=E_{ij}-E_{j+n,i+n} \hspace{.5in} L_i-L_j \mbox{ \,for } i<j$$
$$Y_{ij}:=E_{i,j+n}+E_{j,i+n} \hspace{.5 in} L_i+L_j \mbox{ \,for } i <j$$
$$U_i:=E_{i,i+n} \hspace{.5 in}  2L_i.$$
To simplify notation later in the paper, we abbreviate $X_{ii}:=H_i$, $Y_{ii}:=U_i$.  For $i=1,\ldots n$, we denote the $i$th fundamental weight $\rho_i(H_j):=1$ if $j\leq i$, $0$ else.
%Action of $\mathfrak{sp}_{2n}$ on $\Omega^{even,+}_{closed}(\mathbb{C}^{2n})$ will be used in Section \ref{decomp}, and this action is given by $E_{ij}x_k=\delta_{jk}x_i$, and $E_{ij}dx_k=\delta_{jk}dx_i$.

Recall that the irreducible representations of $\mathfrak{sp}_{2n}$ are parameterized by Young diagrams with at most $n$ rows.  To reduce notational clutter, we will use the same notation, $\lambda=(\lambda_1\geq\ldots \geq \lambda_n)$, to refer both to the Young diagram, and the corresponding irreducible representation of $\mathfrak{sp}_{2n}$.  To avoid confusion, we denote by $\underline{\mu}=[\mu_1\geq \cdots \geq \mu_m]$ the corresponding irreducible representation of $\mathfrak{gl}_m$.  We denote by $Y(k)$ the set of Young diagrams with at most $k$ rows.

Let us recall the restriction functor,
$$\operatorname{Res}^{\mathfrak{gl}_{2n}}_{\mathfrak{sp}_{2n}}:\mathfrak{gl}_{2n}\textrm{-mod} \to \mathfrak{sp}_{2n}\textrm{-mod}.$$
The restriction formula describes the restriction of simple $\mathfrak{gl}_{2n}$ modules; it is as follows. For $\lambda\in Y(n),$ and $\mu\in Y(2n)$, we define:
$$N_{\lambda\mu} =\sum_\eta N_{\eta \lambda \mu},$$
where $N_{\eta \lambda \mu}$ is the Littlewood-Richardson coefficient \cite{FH}, and the sum ranges over partitions $\eta=(\eta_1=\eta_2\geq\eta_3=\eta_4\geq \ldots )\in Y(2n)$ in which each part appears an even number of times.  Then we have:

\begin{thm}[\cite{FH}, p. 427] 
\label{FH thm}
The restriction from $\mathfrak{gl}_{2n}$ to $\mathfrak{sp}_{2n}$ of the representation $\underline{\mu}$ is:
$$\operatorname{Res}^{\mathfrak{gl}_{2n}}_{\mathfrak{sp}_{2n}}(\underline{\mu}) =\bigoplus_{\lambda\in Y(n)} N_{\lambda\mu} \lambda.$$ 
\end{thm}

\subsection{The Lie algebras $W_{2n}$ and $H_{2n}$}
\begin{mydef}
The Lie algebra of polynomial vector fields on $\CC^m$ is $W_m:=\operatorname{Der}\CC[x_1,\ldots x_m]$.
\end{mydef}

\begin{mydef}\label{H2n def}  The Lie algebra of Hamiltonian vector fields on $\CC^{2n}$ is the Lie subalgebra:
$$H_{2n}=\{D\in W_{2n} \,\, | \,\, D \omega=0\},$$ 
of polynomial vector fields that preserve the symplectic form $\omega$.
\end{mydef}

Any $w\in W_m$ may be written $w=\sum_i f_i\frac{\partial}{\partial x_i}$, with $f_i\in \CC[x_1,\ldots,x_m]$.  Likewise, an arbitrary element of $H_{2n}$ can be written in the form:
$$D_u:= \sum_{i=1}^n \left(\frac{\partial u}{\partial 
x_i}\frac{\partial}{\partial x_{i+n}} - \frac{\partial u}{\partial 
x_{i+n}}\frac{\partial}{\partial x_i}\right)$$
The commutator relation is $[D_u,D_v]=D_{\{u,v\}}$, where $\{u,v\}$ is the Poisson bracket,
$$\{u,v\}=\sum_{i=1}^{n} \left(\frac{\partial u}{\partial x_i} \frac{\partial v}{\partial x_{i+n}}-\frac{\partial v}{\partial x_i} \frac{\partial u}{\partial x_{i+n}}\right).$$

We define a grading on $W_m$ by letting $W_m^k$ be the span of vector fields with $\deg f_i=k+1$ for each $i=1,\ldots, m$.  This grading is inherited by $H_{2n}$; it is easy to check that $H_{2n}^k$ is spanned by elements $D_u$ with $\deg u = k+2$.  We let $W_m^{\geq k}:=\oplus_{j\geq k} W_m^j$ and $H_{2n}^{\geq k}:=\oplus_{j\geq k} H_{2n}^j$ denote the trivially induced filtrations.  We may identify $W_m^0$ and $H_{2n}^0$ with $\mathfrak{gl}_{m}$ and $\mathfrak{sp}_{2n}$, respectively.

\subsection{Tensor field modules}\label{tens-field-sec}
We have projections $H_{2n}^{\geq 0}\to H_{2n}^{0}\cong \mathfrak{sp}_{2n}$, and $W_m^{\geq 0} \to W_m^0\cong \mathfrak{gl}_m$, through which we can pull back representations.  In this way we can define the functors of co-induction:
$$\operatorname{CoInd}_{\mathfrak{sp}_{2n}}^{H_{2n}}: \mathfrak{sp}_{2n}\textrm{-mod} \to H_{2n}\textrm{-mod},$$ 
$$ V \mapsto \operatorname{Hom}^{fin}_{H_{2n}^{\geq 0}}(U(H_{2n}), V),$$
$$\operatorname{CoInd}_{\mathfrak{gl}_{m}}^{W_{m}}: \mathfrak{gl}_{m}\textrm{-mod} \to W_{m}\textrm{-mod},$$ 
$$V \mapsto \operatorname{Hom}^{fin}_{W_{m}^{\geq 0}}(U(W_{m}), V),$$
where $\operatorname{Hom}^{fin}$ denotes the homomorphisms with finite-dimensional support, and $U$ is the universal enveloping algebra.

\begin{mydef}\label{f lambda def}
For $\lambda$ (resp. $\underline{\mu}$) an irreducible representations of $\mathfrak{sp}_{2n}$ (resp. $\mathfrak{gl}_{m}$), we define:
$$\mathcal{F}_\lambda:= \operatorname{CoInd}_{\mathfrak{sp}_{2n}}^{H_{2n}} \lambda,$$
$$\mathcal{G}_{\underline{\mu}}:= \operatorname{CoInd}_{\mathfrak{gl}_{2n}}^{W_{2n}} \underline{\mu}.$$
\end{mydef}

We have $H_{2n}^{-1}=\CC\partial_1 \oplus \cdots \oplus \CC\partial_{2n}$, and $W_{m}^{-1}=\CC\partial_1 \oplus \cdots \oplus \CC\partial_{m}$ .  We thus obtain isomorphisms,
$$\mathcal{F}_\lambda \cong 
(\CC[[x_1,\ldots x_{2n}]]\otimes \lambda)^{fin} \cong \CC[x_1,\ldots 
x_{2n}]\otimes \lambda$$
$$\mathcal{G}_{\underline{\mu}} \cong 
(\CC[[x_1,\ldots x_{m}]]\otimes \underline{\mu})^{fin} \cong \CC[x_1,\ldots 
x_{m}]\otimes \underline{\mu}$$

\begin{thm}[\cite{R75}, p. 478]\label{rud rep}
If $\lambda\ne (1^k)$, then $\mathcal{F}_\lambda$ is irreducible.
\end{thm}
\begin{proof} In \cite{R75}, it is proven that the induced modules, 
$$\operatorname{Ind}_{\mathfrak{sp}_{2n}}^{H_{2n}}V_\lambda:= U(H_{2n})\otimes_{H_{2n}^{\geq 0}} V_\lambda,$$
are irreducible for $\lambda \neq (1^k)$.  We have the duality pairing between induction and co-induction:
$$\operatorname{Hom}_{H_{2n}^{\geq 0}}^{fin}(U(H_{2n}),V^*) \otimes (U(H_{2n})\otimes_{H_{2n}^{\geq 0}} V) \to \CC$$
$$f \otimes (a\otimes b) \mapsto \langle f(a),b \rangle.$$
This implies that $\mathcal{F}_\lambda \cong \operatorname{CoInd}_{\mathfrak{sp}_{2n}}^{H_{2n}}V_\lambda^*$ is irreducible.
\end{proof}
The modules $\mathcal{F}_{(1^k)}$ are not irreducible; to describe their structure, we begin by realizing them as submodules in $\Omega^k(\CC^{2n})$, as follows.  Recall from symplectic Hodge theory (see \cite{G} for details) that $\Omega^\bullet(\CC^{2n})$ carries an action of the Lie algebra $\mathfrak{sl}_2$, where $E$ acts by contraction, $\iota_\pi$, with the generating Poisson bi-vector, $F$ acts by wedging with $\omega$, and $H$ acts diagonally: $H\eta = (n-k)\eta$, for $\eta\in \Omega^k(\CC^{2n})$.

The $\mathfrak{sl}_2$-action clearly commutes with the $H_{2n}$-action, and $\mathcal{F}_{(1^k)}$ is the space of $k$-forms lying in the kernel of $\iota_\pi$, i.e. the subspace of $\mathfrak{sl}_2$-singular vectors of weight $n-k$.  In addition to the differential $d$, we have the operator $\delta$ defined by the equation, $\delta =\ast d\ast$, where $\ast$ is the symplectic Hodge star operator.

We define submodules $X_k,Y_k,Z_k,T_k$ of $\mathcal{F}_{(1^k)}$ as follows.  The operators $d$ and $\delta$ generate a copy of $V_{(1)}$ under the $\mathfrak{sl}_2$-action, where $V_{(i)}$ is the $i+1$ dimensional irreducible representation of $\mathfrak{sl}_2$. Thus, upon restricting the map $d:\Omega^k\to \Omega^{k+1}$ to the $V_{(n-k)}$-isotypic component, we find that its image lies in a submodule isomorphic to $(V_{(1)}\otimes V_{(n-k)})$.  Recalling the isomorphism $V_{(1)}\otimes V_{(n-k)}\cong V_{(n-k-1)}\oplus V_{(n-k+1)}$ (for $k < n$), we define:
\begin{align*}
X_k &= \{ x\in \mathcal{F}_{(1^k)} \,\,|\,\, dx=0 \},\\
Y_k &= \{ x\in \mathcal{F}_{(1^k)} \,\,|\,\, dx \textrm{ generates $V_{(n-k-1)}$ or 0 under the $\mathfrak{sl}_2$-action}\},\\
Z_k &= \{ x\in \mathcal{F}_{(1^k)} \,\,|\,\, dx \textrm{ generates $V_{(n-k+1)}$ or 0 under the $\mathfrak{sl}_2$-action}\},\\
T_k &= Y_k + Z_k.
\end{align*}
\begin{rem}Note that $T_k$ is a proper submodule of $\mathcal{F}_k$ because $$T_k \cap \ker d \neq Y_k \cap \ker d + Z_k \cap \ker d.$$ 
As the operators $d,\delta$ commute with $H_{2n}$, the subspaces $X_k,Y_k,Z_k,T_k$ are $H_{2n}$-submodules.\end{rem}
\begin{thm}\label{jhforfk}  For $1\leq k \leq n-1$, the complete lattice of $H_{2n}$-submodules for $\mathcal{F}_{(1^k)}$ is as follows:
$$\xymatrix{
&& Y_k \ar@{^{(}->}[dr] & \\
0\ar@{^{(}->}[r]& X_k \ar@{^{(}->}[dr] \ar@{^{(}->}[ur] & & T_k \ar@{^{(}->}[r]& \mathcal{F}_{(1^k)} \\
&&  Z_k \ar@{^{(}->}[ur]. & \\
}$$

In particular, the Jordan-H\"older series of $\mathcal{F}_{(1^k)}$ is $X_k$, $Y_k/X_k$, $Z_k/X_k$, $\mathcal{F}_{(1^k)}/T_k$.
\end{thm}
\begin{proof}  It is straightforward to check that each containment in the theorem is proper.  In \cite{R75}, it is proven that the induced modules, $\mathcal{F}^*_{(1^k)}$, have length four; thus so do the $\mathcal{F}_{(1^k)}$, and the theorem follows.
\end{proof} 
\begin{thm}\label{extracases} For $k=0,n$, $\mathcal{F}_{(1^k)}$ has the closed forms $X_k$ as a sub-module; the quotient is irreducible. \end{thm}
\begin{proof}
For $k=0$, it well-known that functions modulo constants form an irreducible $H_{2n}$-module.  For $k=n$, it is shown in \cite{R75} that $\mathcal{F}_{(1^n)}$ has length two, as an $H_{2n}$-module; the closed forms clearly form a proper submodule.
\end{proof}
%If $\lambda=(1^k)$, $\mathcal{F}_\lambda$ has a submodule $F_k$ such that $\mathcal{F}_\lambda/F_k$ is irreducible, and $F_k$ is generated by the singular vectors:
%\begin{align*}
%&\bar{x}_k\in S^0[x_1,\ldots , x_{2n}]\otimes \Gamma_\lambda, \mbox{ weight  }(1^k)\\
%&\bar{y}_k \in S^1[x_1,\ldots , x_{2n}]\otimes \Gamma_\lambda, \mbox{ weight  } (1^{k-1})\\
%&\bar{z}_k \in S^1[x_1,\ldots , x_{2n}]\otimes \Gamma_\lambda, \mbox{ weight  }(1^{k+1})\\
%&\bar{t}_k \in S^2[x_1, \ldots , x_{2n}]\otimes \Gamma_\lambda, \mbox{ weight  } (1^{k})
%\end{align*}
%We have the inclusions $\bar{x}_k\in U(H_{2n})\cdot \bar{y}_k$ and $\bar{x}_k\in U(H_{2n})\cdot \bar{z}_k$
%\end{thm}

%\begin{thm}\label{subquotientthm}
%Let $F_{\bar{x}_k}$, $F_{\bar{y}_k}$, $F_{\bar{z}_k}$, $F_{\bar{t}_k}$ denote the submodules generated by the corresponding singular vectors. Then $\bar{F}_{\bar{x}_k}=F_{\bar{x}_k}/(F_{\bar{y}_k}+F_{\bar{z}_k})$, $\bar{F}_{\bar{y}_k}=F_{\bar{y}_z}/F_{\bar{t}_k}$, $\bar{F}_{\bar{z}_k}=F_{\bar{z}_k}/F_{\bar{t}_k}$, and $\bar{F}_{\bar{t}_k}=F_{\bar{t}_k}$ are irreducible.
%\end{thm}

It follows from Definition \ref{f lambda def} that restriction and coinduction commute:
$$\operatorname{Res}_{H_{2n}}^{W_{2n}}\operatorname{CoInd}_{\mathfrak{gl}_{2n}}^{W_{2n}}V \cong \operatorname{CoInd}_{\mathfrak{sp}_{2n}}^{H_{2n}}\operatorname{Res}_{\mathfrak{sp}_{2n}}^{\mathfrak{gl}_{2n}}V.$$

In particular, we have from Theorem \ref{FH thm} the restriction formula:

\begin{equation}\label{rest-form}\mbox{Res}^{W_{2n}}_{H_{2n}}(\mathcal{G}_{\underline{\mu}}) \cong \bigoplus_\lambda N_{\lambda\mu} \mathcal{F}_{\lambda}.\end{equation}

\subsection{Distinguished cyclic vectors}

Recall that the tensor field modules $\mathcal{F}_{\lambda}$ are irreducible when $\lambda\neq(1^k)$, and that we have an isomorphism of $\mathfrak{sp}_{2n}$-modules,
$$\mathcal{F}_\lambda \cong 
(\CC[[x_1,\ldots x_{2n}]]\otimes \lambda)^{fin} \cong \CC[x_1,\ldots 
x_{2n}]\otimes \lambda.$$

In particular, we observe that $\mathcal{F}_\lambda$ contains a unique-up-to-scalars $\mathfrak{sp}_{2n}$-highest weight vector $v_\lambda$ of weight $\lambda$.  The following proposition follows immediately:

\begin{prop}
Let $\lambda\neq (1^k) \in Y(n)$.  A tensor field module $M$ for $H_{2n}$ is isomorphic to $\mathcal{F}_\lambda$ if, and only if, $M$ contains a highest weight vector $v_\lambda$ for $\mathfrak{sp}_{2n}$, such that $\partial_i v_\lambda=0$ for all $i=1,\ldots, 2n$.
\end{prop}

We call any such highest weight vector $v_\lambda \in \mathcal{F}_\lambda$ a distinguished cyclic vector.  Note that, in \cite{R75}, the term ``singular vector" is used in the dual context.

The situation for $\mathcal{F}_{(1^k)}$ is more complicated.
\begin{prop}\label{partialprop}  We have:
\begin{enumerate}
%\item The $H_{2n}$-module $\mathcal{F}_{(1^k)}\subset \Omega^k(\CC^{2n})$ is generated by $t_k=x_1dx_1\cdots dx_k$.
\item The $H_{2n}$-submodule $X_k\subset \mathcal{F}_{(1^k)}\subset \Omega^k(\CC^{2n})$ is generated by: 
$$\x_k=dx_1\cdots dx_k.$$
\item The $H_{2n}$-submodule $Y_k\subset \mathcal{F}_{(1^k)}\subset \Omega^k(\CC^{2n})$ is generated by:
$$\y_k=\sum_{i=1}^{k+1}(-1)^{i}x_{i}dx_1\cdots\widehat{dx_i}\cdots dx_{k+1}.$$
\item The $H_{2n}$-submodule $Z_k\subset \mathcal{F}_{(1^k)}\subset \Omega^k(\CC^{2n})$ is generated by:
$$\z_k=\omega\wedge\y_{k-2} - 2(n-k+2)y\wedge\x_{k-1}.$$  
where $y=\frac 1 2 \sum_i (x_idx_{i+n}-x_{i+n}dx_i)$ is the Liouville form of $\omega$.
\end{enumerate}
\end{prop}

\begin{proof}
It is well-known that each of $X_k, Y_k, Z_k$ is generated as $H_{2n}$ by the unique $\mathfrak{sp}_{2n}$-highest weight vector of weight $(1^k), (1^{k+1}), (1^{k-1})$ and degree $k, k+1, k+1$, respectively.  It is straightforward to check that each of $\x_k, \y_k$, and $\z_k$ is highest weight of the correct weight and degree, and that the $\x_k$, $\y_k$ are in $X_k$, $Y_k$, respectively.  It remains only to show that $\z_k\in Z_k$, which follows from the identities:

$$i_\pi(\omega\y_{k-2}) = (n-k+2)\y_{k-2},\quad i_\pi(y \x_{k-1})=\frac{\y_{k-2}}{2},\quad$$
$$i_\pi(d\z_k)=(n-k+1)(k-3-2n)\x_{k-1}\neq 0.$$
\end{proof}
\begin{prop} We have the following identities:
\begin{align*}
\partial_l\y_k &= \left\{\begin{array}{ll}(-1)^{k} X^T_{l,k+1}\x_k, & 1\leq l \leq k,\\ (-1)^{k+1}\x_k, & l=k+1,\\ 0, & k+2\leq l \leq 2n. \end{array}\right.\\
\partial_l\z_k &= \left\{\begin{array}{ll}
(-1)^k\sum_{j\geq k}Y^T_{lj}X_{kj}^T\x_k,& 1\leq l \leq k-1,\\
(-1)^k(n-k+2)Y^T_{k,l}\x_k , & k\leq l \leq n,\\
0,& 1 \leq l-n\leq k-1,\\(-1)^{k-1}(n-k+2)X_{k,l-n}^T\x_k, &k \leq l-n\leq n.\end{array}\right.
\end{align*}
\end{prop}

\begin{proof}
We compute:
{\footnotesize
\begin{align*}
\partial_l\y_k &= \left\{\begin{array}{ll}(-1)^l dx_1\cdots \widehat{dx_l} \cdots dx_{k+1}, & 1\leq l \leq k,\\ (-1)^{k+1}dx_1\cdots dx_k, & l=k+1,\\ 0, & k+2\leq l \leq 2n. \end{array}\right.\\
\partial_l\z_k &= \left\{\begin{array}{ll}
(-1)^l\omega dx_1\cdots \widehat{dx_l} \cdots dx_{k-1} - (n-k+2) dx_{n+l} dx_1\cdots  dx_{k-1},& 1\leq l \leq k-1,\\
-(n-k+2)dx_{l+n}dx_1\cdots dx_{k-1} , & k\leq l \leq n,\\
0,& 1 \leq l-n\leq k-1,\\dx_1\cdots dx_{k-1} dx_{l-n}, &k \leq l-n\leq n.\end{array}\right.
\end{align*}}
The claims now follow by direct comparison.
\end{proof}

%
%\section{Action of $H_{2n}$ on $B_k(A'_{2n})$}\label{action}
%In this section, we construct an action of $H_{2n}$ on each $B_k(A'_{2n})$, echoing that of $W_{2n}$ on each $B_k(A_{2n})$.
\section{Action of $H_{2n}$ on $B_k(A'_{2n})$}\label{action}
In this section, we construct an action of $H_{2n}$ on each $B_k(A'_{2n})$, echoing that of $W_{2n}$ on each $B_k(A_{2n})$.

\begin{lma}\label{actionlemma}
Let $I$ be an ideal of an associative algebra $A$, and let $\operatorname{Im}(L_k(A)\cap I)$ denote the image of $L_k(A)\cap I$ in $B_k(A)$ then 
$$B_k(A/I)=B_k(A)/\operatorname{Im}(L_k(A)\cap I).$$
\end{lma}

\begin{proof}

Lemma 2.4 of \cite{BB} states the result in the case $k=2$, while the same proof applies to all $k$. 
\end{proof}

\begin{prop}
The $W_{2n}$ action on each $B_k(A_{2n})$ and $N_k(A_{2n})$ descends to an action of $H_{2n}$ on $B_k(A'_{2n})$ and $N_k(A'_{2n})$. 
\end{prop}

\begin{proof}
From Lemma \ref{actionlemma}, $B_k(A'_{2n})=B_k(A_{2n})/(L_k(A_{2n})\cap \langle\omega\rangle)$. Since both $L_k(A_{2n})$ and $\langle \omega \rangle$ are invariant under action by $H_{2n}$, so is their intersection; thus, the action descends. 
\end{proof}

\begin{cor} \label{B2 C iso} We have the following isomorphisms:
\begin{enumerate}
\item $B_2(A'_{2n})\cong \Omega^{even,+}_{closed}(\mathbb{C}^{2n})/ \langle \omega \rangle$.
\item $A'_{2n}/M_3(A'_{2n}) \cong \Omega^{even}_{*}(\mathbb{C}^{2n})/\langle \omega \rangle$.
\item $\bar{B}_1(A'_{2n})\cong (\Omega^{even}(\mathbb{C}^{2n})/\Omega^{even}_{closed}(\mathbb{C}^{2n}))/\langle \omega \rangle$.
\end{enumerate}
\end{cor}

\begin{proof}
Recall that under the Fedosov product, $\omega\in A_{2n}$ maps to the standard symplectic form on $\mathbb{C}^{2n}$, which we also denote as $\omega$. For (1), let $A=A_{2n}$, $I=\langle \omega \rangle$ in Lemma \ref{actionlemma}, which gives $B'_{2,2n}\cong B_{2,2n}/(L_2 \cap \langle \omega \rangle)$. The result follows from the isomorphism given in \cite{FS}, $B_{2,2n}\cong \Omega^{even,+}_{closed}(\mathbb{C}^{2n})$. The proof of (3) follows similarly.  For (2), $A_{2n}/M_3 \cong \Omega^{even}_{*}(\mathbb{C}^{2n})$ as shown in \cite{FS}. The result follows by application of the third isomorphism theorem.
\end{proof}

\subsection{Finite length Jordan-H\"older series}

\begin{prop}
The $H_{2n}$-module Jordan-H\"older series of $B_k(A'_{2n})$ is finite length, and is composed of $\mathcal{F}_\lambda$, and of the irreducible submodules of the reducible $\mathcal{F}_{(1^k)}$.  
\end{prop}

\begin{proof}
The main result of \cite{DE} asserts that the Jordan-H\"older series of $B_k(A_{2n}),$ for $k\geq 3$ as a $W_{2n}$-module consists of tensor field modules $\mathcal{G}_{\underline{\mu}}$ for a finite set of $\underline{\mu}\in Y_{2n}$.  Thus the Jordan-H\"older series of $\operatorname{Res}^{W_{2n}}_{H_{2n}}B_k(A_{2n})$ consists of a finite set of $\mathcal{F}_\lambda$, according to the restriction rules \eqref{rest-form}, which are finite-to-one. For all $\lambda$ except $\lambda=(1^k)$, $\mathcal{F}_\lambda$ is irreducible, while each $\mathcal{F}_{(1^k)}$ has finite length Jordan-H\"older series by Theorems \ref{jhforfk}, \ref{extracases}; thus $B_k(A_{2n})$ is of finite length. Thus when we quotient by $\langle \omega \rangle$, we find that $B_k(A'_{2n})$ also has finite length.
\end{proof}

\section{Structure of $A_{2n}'/M_3$, $\bar{B}_1(A_{2n}')$, and $B_2(A'_{2n})$}\label{Bkdescs}

Recall from Equation \eqref{rest-form} that the space of $k$-forms, $\Omega^k(\CC^{2n})\cong \mathcal{G}_{(1^k)}$ decomposes as an $H_{2n}$-module:
$$\Omega^k(\CC^{2n})\cong \bigoplus_{s\leq \frac{k}{2}} \mathcal{F}_{(1^{k-2s})},$$
and that each $\mathcal{F}_{(1^{k-2s})}$, for $s\geq 1$, consists of forms divisible by $\omega$.  On the other hand, $\mathcal{F}_{(1^k)}$ consists of locally finite $\mathfrak{sl}_2$-singular vectors of weight $n-k$, and so cannot lie in the image of $F= \omega\wedge -$.  Thus we have $\Omega^k(\CC^{2n})/\langle \omega \rangle \cong \mathcal{F}_{(1^k)}$.

\begin{prop} For $\leq k\leq n-1$, we have $Z_k = \mathcal{F}_{(1^k)}\cap\langle \omega \rangle + X_k$. \end{prop}
\begin{proof}
First, we give an alternate description of $Z_k$.  Consider the $H_{2n}$-submodule $\widetilde{Z}_k\subset\mathcal{F}_{(1^k)}$, consisting of those $\alpha$ satisfying $d\alpha \in \Omega^{k-1}_{ex}\cdot\langle \omega \rangle.$  Clearly $\widetilde{Z}_k \supsetneq X_k$, and $\widetilde{Z}_k \cap Y_k = X_k$.  According to Theorem \ref{jhforfk}, $Z_k$ is the unique such $H_{2n}$-submodule, and so $\widetilde{Z_k}=Z_k$.

The containment $\supseteq$ is clear, using the alternate description of $Z_k$.  Conversely, suppose that $d\alpha=(d\nu)\omega.$  Then, integration by parts gives:
$$\alpha=\nu\omega -\nu d\omega + d\eta = \nu\omega +d\eta,$$ for some exact form $d\eta$.
\end{proof}
\begin{cor}
We have the following:
\begin{align*}
A_{2n}'/M_3 &\cong \bigoplus_{\underset{0\leq k \leq n}{k \textrm{ even}}} \mathcal{F}_{(1^k)}\\
\bar{B}_1(A_{2n}')&\cong \mathcal{F}_{(0)}\oplus\bigoplus_{\underset{2\leq k\leq n}{k \textrm{ even}}} \mathcal{F}_{(1^k)}/Z_k\\
B_2(A_{2n}')&\cong \bigoplus_{\underset{2\leq k\leq n}{k \textrm{ even}}} Z_k.
\end{align*}
\end{cor}

Theorems \ref{AmodM3thm}, \ref{barB1}, and \ref{B2} follow from the corollary, and Theorems \ref{jhforfk} and \ref{extracases}.
 
\section{Structure of $B_3(A_{2n}')$}\label{B3sec}
In this section, we give a complete description of $B_3(A_{2n})$ as an $H_{2n}$-module, and use this to conjecture a description of $B_3(A'_{2n})$.  Theorem 1.8 of \cite{AJ} gives the following decomposition:
$$B_3(A_{2n})\cong\bigoplus_{i=1}^{n} \mathcal{G}_{\underline{(2,1^{2i-1})}}.$$

By Theorem \ref{FH thm}, for $k\le n-1$, odd, we have:
$$\operatorname{Res}^{W_{2n}}_{H_{2n}}\mathcal{G}_{\underline{(2,1^k)}}=\bigoplus_{s=0}^{(k+1)/2} \left( \mathcal{F}_{(2,1^{k-2s})} \oplus \mathcal{F}_{(1^{k-2s})} \right)$$ 

%We now give highest weight vectors for each of these representations, starting with $(2,1^{k-2s})$. From this the decomposition of $B'_{3,2n}$ will be clear. We first need the following definitions:

\begin{rem}
Recall that we have a surjection $A/M_3 \otimes B_2 \to B_3$ given by $a\ot b\mapsto [a,b]$, relying on the containment $[M_3,L_2]\subset L_4$, proved in \cite{FS}.  We have the Feigin-Shoikhet isomorphisms $A_{2n}/M_3\cong \Omega^{ev}(\CC^{2n})$, $B_2(A_{2n})\cong \Omega^{+,ev}_{ex}(\CC^{2n})$.  Thus we have a surjection,
$$\Omega^{ev}(\CC^{2n})\ot \Omega^{ev}_{ex}(\CC^{2n}) \to B_3(A_{2n}).$$
We abuse notation and write $[a,b]$ for the image of $a\ot b$.
\end{rem}

\begin{comment}
We define the linear operator,
$$W:A_{2n}\rightarrow A_{2n}$$
$$W(v):=\frac{1}{4}\sum_{i=1}^n [x_i,x_{i+n}v] -[x_{i+n},x_iv].$$
We denote $W^k(v):=W(W^{k-1}(v))$, $W^0(v)=v$.

\begin{rem} The operator $W$ preserves $M_3$, and so descends to an operator on $A_{2n}/M_3$.  Under the Feigin-Shoikhet isomorphism, an easy computation yields:
$$W(v)=\frac12(\omega v + y dv).$$
In particular, $W$ is an $H_{2n}$-module endomorphism of $A_{2n}/M_3$.
\end{rem}
\end{comment}

%The following lemma will be of use. Its proof is a simple computation:
% \begin{lma}
% For $X\in \mathfrak{sp}_{2n}$, if $X(v)=0$, then $X(W(v))=0$
% \end{lma}
 
\begin{comment}
We introduce a multi-linear product, $\{-,\ldots,-\}: A_{2n}^{\otimes k}\to A$, given inductively by:
$$\{a_1\}=a_1, \quad\{a_1,a_2\}:=[a_1,a_2],$$
$$\{a_1,\ldots, a_k\}:=[a_1,a_2\cdot\{a_3,a_4,\ldots,a_k\}]$$

For example, we have:
$$\{x_1,x_2,x_3,x_4\}=[x_1,x_2[x_3,x_4]],\quad \{x_1,x_2, \omega\}=[x_1,x_2\omega].$$
Note that $\{-,\ldots,-\}$ preserves the ideal $M_3$, and so descends to $A_{2n}/M_3$.  The image of $\{a_1,\ldots, a_{2k}\}$ under the Feigin-Shoikhet isomorphism is a multiple of $da_1\wedge \ldots \wedge da_{2k}$.
\end{comment}

\begin{lma}\label{notdcvgenerator}
For $k$ odd, and $0\leq s\leq (k-1)/2$, the submodule
$$\mathcal{F}_{(2,1^{k-2s})}\subset \operatorname{Res}^{W_{2n}}_{H_{2n}}\mathcal{G}_{\underline{(2,1^k)}} \subset B_3(A_{2n}),$$
is generated by $v_{k,s}=[x_1,\xi],$ where $\xi=dx_1\wedge \ldots \wedge dx_{k-2s+1} \omega^s$
%$\xi=\{x_1,x_2,\ldots ,x_{k-2s+1},W^{s}(1)\}$.
\end{lma}

\begin{proof}
%The image of $\xi$ in $A_{2n}/M_3$, under the Feigin-Shoikhet isomorphism, is the closed form:
%$$\xi=dx_1\wedge \cdots \wedge dx_{k-2s}w^s.$$
% = dx_1\wedge\cdots\wedge dx_{k-s}\wedge dy_{k-2s+1}\wedge \cdots \wedge dy_{k-s}.$$
Notice $\xi$ is an even, closed form, so that we have $v_{k,s} \in L_3$.  The vectors $x_1$ and $\xi$, and thus $x_1\otimes \xi$, are clearly highest weight vectors for $\mathfrak{sp}_{2n}$.  We have $\partial_iv_{k,s}=0,$ for any $i=1,\ldots, 2n$.

It only remains to show that $v_{k,s}\not\in L_4$.  Notice that 
$$2^{-(k-2s+1)/2}[x_1,x_2[x_3,x_4[\cdots [x_{k-2s},x_{k-2s+1}\omega^s]\cdots]$$
maps to $\xi$ under the Feigin-Shoikhet isomorphism.
The proof is very similar to that of Proposition 5.11 of \cite{AJ}: we find an algebra $B$, and a map $\theta:A_{2n}\to B$ in which we can compute directly that $\theta(v_{k,s})\not\in L_4(B)$. We let $B=A\otimes E$, where $A$ is the free algebra on two generators $a,b$, E is the exterior algebra with generators $z_0,\ldots z_{2n}$. We define $\theta$ by $\theta(x_1)=ez_0+fz_1$, and $\theta(x_i)=z_i$ for $i\ge2$.

We compute:
\begin{align*}
\theta(\omega)&=ez_0z_{1+n}+fz_1z_{1+n}+\sum_{i\ge 2} z_iz_{i+n},\\
%\theta(W(\omega))&=2(az_0z_{1+n}+bz_1z_{1+n})\left(\sum_{i\ge 2} z_iz_{i+n}\right)+\left(\sum_{i\ge 2} z_iz_{i+n}\right)^2\\
%&\vdots \\
\theta(\omega^s)&=s(ez_0z_{1+n}+fz_1z_{1+n})\left(\sum_{i\ge 2} z_iz_{i+n}\right)^{s-1}+\left(\sum_{i\ge 2} z_iz_{i+n}\right)^s\\
\theta(v_{k,s})&=4[e,f]z_0\ldots z_{k-2s+1}\left(\sum_{i\ge 2} z_iz_{i+n}\right)^s.
\end{align*}

By applying Corollary 5.10 of \cite{AJ}, we see that $v_{k,s}$ is nonzero in $B_3(A_{2n})$, and thus is a distinguished cyclic generator of $\mathcal{F}_{(2,1^{k-2s})}$.\end{proof}  

\begin{cor}All summands $\mathcal{F}_{(2,1^{k-2s})}\subset \operatorname{Res}^{W_{2n}}_{H_{2n}}\mathcal{G}_{\underline{(2,1^k)}}$, except $\mathcal{F}_{(2,1^k)}$, are zero in $B_3(A_{2n}')$.\end{cor}
\begin{proof}
Clearly $v_{k,s}\in \langle \omega \rangle$ if, and only if, $s>0$. 
\end{proof}

To find the singular vectors corresponding to summands $\mathcal{F}_{(1^l)}$, inside $B_3(A_{2n})$, we introduce the following homomorphisms of $\mathfrak{sp}_{2n}$-modules:
\begin{align*}&\phi_s: \Omega^{ev} \ot \Omega^{odd}_{ex} \to B_3(A_{2n}),\\
&u\ot v\mapsto \sum_i [x_iu , v dx_{i+n} \omega^s] - [x_{i+n}u,vdx_i\omega^s],\\
&\psi_s: \Omega^{ev} \to B_3(A_{2n}),\\
&v\mapsto \sum_i [\omega^sx_{i+n},d(vx_i)] - [\omega^sx_i,d(vx_{i+n})].\end{align*}
We note that constant vector fields do not commute with $\phi_s, \psi_s$; rather, we have:
\begin{align*}
\partial_l\phi_s(u\ot v) &= \left\{ \begin{array}{ll}\phi_s(\partial_l(u\ot v)) + [u, v dx_{l+n}\omega^s], & 1\leq l \leq n\\
\phi_s(\partial_l(u\ot v)) - [u, vdx_{l-n}\omega^s], & n+1\leq l \leq 2n\end{array}\right.\\
\partial_l\psi_s(v) &= \left\{\begin{array}{ll}\psi_s(\partial_lv) + [\omega^sx_{l+n},dv],& 1\leq l \leq n\\
\psi_s(\partial_lv) - [\omega^sx_{l-n},dv],& n+1\leq l \leq 2n\end{array}\right..
\end{align*}
We also define the following elements of $\Omega(\CC^{2n})$:
\begin{align*}
a_k&=dx_1\wedge \ldots \wedge dx_k\\
p_{j,m}&=(-1)^jdx_1 \wedge \ldots \wedge \hat{dx}_j \wedge \ldots \wedge dx_m, \textrm {or $0$ when $m<j$.} \\
q_{m}&=\sum_{j=1}^m  x_j p_{j,m}
\end{align*}
We collect here several easily proven observations for later use:
\begin{prop} The vector $q_m$ is a $\mathfrak{sp}_{2n}$-highest weight vector of weight $\rho_m$, and we have the following identities:
$$p_{m,m}=(-1)^m a_{m-1},\quad dq_m=m a_m, \quad \partial_l q_m=\left\{\begin{array}{ll} p_{l,m},& \textrm{if $l\leq m$,} \\ 0, &\textrm{otherwise.}\end{array}\right.$$
\end{prop}

%\begin{lma}\label{plma}
%For each $l$ $q_l$ is a highest weight vector of weight $l$.
%\end{lma}
% 
%\begin{proof}
%It is trivial that $U_i$ kills $q_{l,s}$. For $X_{ik}$ we make the following computation:
%\begin{align*}
%X_{ik}(q_{l})&=(-1)^k x_i dx_1\wedge \ldots \wedge \hat{dx}_k\wedge \ldots \wedge dx_l\\
%&+(-1)^i x_j dx_1\wedge \ldots \wedge \hat{x}_i\wedge \ldots \wedge \hat{dx}_k dx_i\wedge\ldots \wedge dx_l \\
%&= ((-1)^k+(-1)^i(-1)^{l-i+1})dx_1\wedge \ldots \wedge \hat{dx}_k\wedge \ldots \wedge dx_l\\
%&=0\\
%\end{align*}
%Thus $q_{l}$ is a highest weight vector of weight $l$.
%\end{proof}
 
We now construct distinguished cyclic vectors $\bar{x}_{k,s}$, $\bar{y}_{k,s}$, $\bar{z}_{k,s}$ for the summands $\mathcal{F}_{(1^{k-2s})}\subset (2,1^k)\subset B_3(A_{2n})$.
 
\begin{thm}\label{dcvprop}
Let $1\leq k \leq n$ be an odd integer.  The distinguished cyclic vectors for the $H_{2n}$-submodule,
$$\mathcal{F}_{(1^{k-2s})}\subset \operatorname{Res}^{W_{2n}}_{H_{2n}}\mathcal{G}_{(2,1^k)}\subset B_3(A_{2n}),$$
are given as follows:
\begin{align*}
\bar{x}_{k,s}&=\phi_s(1\otimes a_{k-2s})\\
\bar{y}_{k,s}&=
\left\{\begin{array}{ll}
\displaystyle
\frac{k-2s+1}{k-4s}\phi_s\left(\sum_{j=1}^{k-2s+1} x_j\otimes p_{j,k-2s+1}\right) - \frac{\psi_s(q_{k-2s+1})}{k-4s} & \textrm{if $k-2s\ne n-1$}\\
\displaystyle
\sum_{j=1}^{n} [\tilde{p}_{j,n},x_j\omega^{s+1}] & \textrm{if $k-2s=n-1$}\\
\end{array}\right. \\
\bar{z}_{k,s}&=
\left\{\begin{array}{ll}\displaystyle\sum_{i=1}^{n} [x_i,d(x_{i+n}y)\omega^s]-[x_{i+n},d(x_{i}y)\omega^s],& 
\displaystyle
\textrm{if $k-2s=1$} \\
\displaystyle
\frac{2n-3k+6s+5}{k-2s-1}y_{k,s+1} - \frac{2(n-k+2s+2)}{k-2s-1}\phi_s\left(1\otimes d(q_{k-2s-1}y)\right)& \textrm{if $k-2s\neq1$}\end{array}\right. \\
\end{align*}

\end{thm}

%Keep for reference 
%\bar{z}_{1,0}=\sum_{i,j} [x_i,[x_{i+n}x_{j},x_{j+n}]]-[x_{i},[x_{j+n}x_{i+n},x_{j}]]
%-[x_{i+n}b_k,[x_ix_{j},x_{j+n}]]+[x_{i+n}b,[x_{j+n}x_{i},x_{j}]]
 
The proof of this theorem will comprise the remainder of the present section.  To begin, we collect several observations and lemmas:

\begin{lma}\label{alphalma} For k odd, let:
$$\alpha_{k,s,m} := \sum_j[x_j,p_{j,k-2s+1}dx_{m}\omega^s] - [\omega^sx_{m},a_{k-2s+1}].$$ 
We have $\alpha_{k,s,m}=0 \mod L_4$, for all $k,s,m$.\end{lma}
\begin{proof}

Let $B$ denote the free algebra on generators $z_0, z_1, \ldots, z_{2n}$, and define a surjective homomorphism $B\to A_{2n}$ by:
$$z_0\mapsto \omega^sx_{m}, \qquad z_i\mapsto x_i, \textrm{for $i=1,\ldots,2n$}.$$

We note that each expression $\alpha_{k,s,m}$ is alternating in the generators $z_i$, and lies in the image of $B_3(B)$.  Lemma 5.1 of \cite{AJ} therefore implies that $\alpha_{k,s,m}=0 \mod L_4.$
\end{proof}

\begin{lma}\label{omegaswitch}
For any $a\in A_{2n}$ and $k,l,s$ we have:
$$ l[a\omega^{s-k},\omega^{k}] = k[a\omega^{s-l},\omega^l] \mod L_4.$$
\end{lma}

\begin{proof}
We observe:
$$[a,\omega^k]=\sum_j [\omega^ja\omega^{k-1-j},\omega] = k[a\omega^{k-1},\omega] \mod L_4,$$
where we have used the containment $[M_3,A]\subset L_4$, from \cite{BJ}.
Replacing $a$ by $\omega^{s-k}a$ gives:
$$l[\omega^{s-k}a,\omega^k] = kl[a\omega^{s-1},\omega] = k[\omega^{s-l}a,\omega^l]\mod L_4.$$
\end{proof}

\begin{lma}\label{phipsirelation}
We have $\psi_s(p_{l,k-2s+1})=(2s+1)\phi_s(1\ot p_{l,k-2s+1})$
\end{lma}

\begin{proof}
Let $\tilde{p}_{l,k-2s+1}$ be a form such that $d\tilde{p}_{l,k-2s+1}=p_{l,k-2s+1}$.  We compute:
 \begin{align*}
 \psi_s(p_{l,k-2s+1})&=-\sum_i [\omega^sx_{i},p_{l,k-2s+1}dx_{i+n}]-[\omega^sx_{i+n},p_{l,k-2s+1}dx_i]\\
  &=-\frac12\sum_i [\omega^sx_{i},[\tilde{p}_{l,k-2s+1},x_{i+n}]]-[\omega^sx_{i+n},[\tilde{p}_{l,k-2s+1},x_i]]\\
  &=-\sum_i-[x_i,p_{l,k-2s+1}dx_{i+n}\omega^s]+[x_{i+n},p_{l,k-2s+1}dx_i\omega^s]\\&\phantom{==} +2[\tilde{p}_{l,k-2s+1},\omega^{s+1}]\\
  &=-\phi_s(1\ot p_{l,k-2s+1}) + 2[\tilde{p}_{l,k-2s+1},\omega^{s+1}],
\end{align*}
by the Jacobi identity.  By Lemma \ref{omegaswitch}, we have:
$$[\tilde{p}_{l,k-2s+1},\omega^{s+1}]=(s+1)[\tilde{p}_{l,k-2s+1}\omega^s,\omega].$$
For any $b\in A_{2n}$, the Jacobi identity implies: 
$$[\omega,b]=\frac12\sum_i [x_i,[x_{i+n},b]]-[x_{i+n},[x_i,b]],$$
so that
$$2[\tilde{p}_{l,k-2s+1},\omega^{s+1}]=-2(s+1)\sum_i [x_i,[x_{i+n},\tilde{p}_{l,k-2s+1}\omega^s]]-[x_{i+n},[x_i,\tilde{p}_{l,k-2s+1}\omega^s]].$$
Thus, by converting to differential forms, and exchanging $p_{l,k-2s+1}$ and $dx_i$ terms, we find:
\begin{align*}\psi_s(p_{l,k-2s+1})&=(2s+1)\sum_i [x_i,p_{l,k-2s+1}dx_{i+n}\omega^s]-[x_{i+n},p_{l,k-2s+1}dx_i\omega^s]\\&=(2s+1)\phi_s(1\ot p_{l,k-2s+1}),\end{align*}
as desired.
\end{proof}

\begin{lma}
For $k=n$, $\bar{y}_{k,s}\in L_3$
\end{lma}
\begin{proof}
Notice that the image of $y_{n,s}$ is 0 under the Fedosov product, thus 
$$y_{n,s}\in L_2 \cap M_3=L_3$$
by the key-lemma of \cite{FS}.
\end{proof}

\begin{prop} We have the following:
\begin{enumerate}
\item The vectors $\bar{x}_{k,s}, \bar{y}_{k,s}, \bar{z}_{k,s}$ are $\mathfrak{sp}_{2n}$-highest weight vectors of weight $\rho_{k-2s}$, $\rho_{k+1-2s}$, $\rho_{k-1-2s}$, and total degree $k$, $k+1$, and $k+1$, respectively.
\item The vectors $\bar{x}_{k,s}$, $\bar{y}_{k,s}$, $\bar{z}_{k,s}$ satisfy the same equations as $\x_{k-2s}, \y_{k-2s}, \z_{k-2}$ in Proposition \ref{partialprop}:
\begin{align*}
\partial_l\bar{x}_{k,s}&=0, \textrm { for } 1\leq l \leq 2n,\\
\partial_l\bar{y}_{k,s} &= \left\{\begin{array}{ll}- X^T_{l,k+1}\bar{x}_{k,s}, & 1\leq l \leq k,\\ \bar{x}_{k,s}, & l=k+1,\\ 0, & k+2\leq l \leq 2n. \end{array}\right.\\
%\partial_l\bar{y}_{k,s}&=\left\{\begin{array}{ll}(-1)^{k+1} X^T_{l,k-2s+1}\bar{x}_{k,s}, &1\leq l \leq k-2s\\ (-1)^{k+1}\bar{x}_{k,s}, & l=k-2s+1\\ 0, & k-2s+2\leq l \leq 2n \end{array}\right.\\
%\partial_l\bar{z}_{k,s} &= \left\{\begin{array}{ll} Y_{lk}^T\bar{x}_{k,s}, & 1\leq l \leq k-1\\ U_k^T\bar{x}_{k,s}, &l=k\\ Y_{kl}^T \bar{x}_{k,s}, & k+1\leq l\leq n\\-X_{k,l-n}^T\bar{x}_{k,s},& n+1\leq l \leq k-1\\-\bar{x}_{k,s},& l=n+k\\0,& n+1\leq l \leq 2n\end{array}\right.
\partial_l\bar{z}_{k,s} &= \left\{\begin{array}{ll}
-\sum_{j\geq k}Y^T_{lj}X_{kj}^T\bar{x}_{k,s},& 1\leq l \leq k-1,\\
-(n-k+2)Y^T_{k,l}\bar{x}_{k,s} , & k\leq l \leq n,\\
0,& 1 \leq l-n\leq k-1,\\(n-k+2)X_{k,l-n}^T\bar{x}_{k,s}, &k \leq l-n\leq n.\end{array}\right.
\end{align*}
\item The vectors $\bar{x}_{k,s}, \bar{y}_{k,s}, \bar{z}_{k,s}$ are the unique (up to scalars) vectors in $B_3(A_2)$ satisfying (1) and (2).
\end{enumerate}
\end{prop}

\begin{proof} Claim (1) follows from the fact that $\phi_s$ is a homomorphism of $\mathfrak{sp}_{2n}$-modules, together with the observation that all arguments of $\phi_s$ are clearly highest weight of the asserted weights.  Claim (2) for $\bar{x}_{k,s}$ follows from the same claim for $a_{k-2s}$, which is clear.  For (2), we compute for $k\ne n-1$ (the claim is clear for $k=n-1$):

\noindent For $1\leq l \leq k-2s+1$:
\begin{align*}
\partial_l\bar{y}_{k,s} &= \frac{k-2s+1}{k-4s}\left(\sum_j\phi_s(\partial_l(x_j\ot p_{j,k-2s+1})  +  [x_j,p_{j,k-2s+1}dx_{l+n}\omega^s]\right.)\\
&\left. \phantom{===} - \frac{\psi_s(p_{l,k-2s+1})}{k-2s+1} - [\omega^sx_{l+n},a_{k-2s+1}]\right)\\
&=\frac{k-2s+1}{k-4s}\left(\phi_s(1\ot p_{l,k-2s+1}) - \frac{\psi_s(p_{l,k-2s+1})}{k-2s+1} + \alpha_{k,s,l+n}\right)\\
&=\phi_s(1\ot p_{l,k-2s+1}),
\end{align*}
by Lemmas \ref{alphalma} and \ref{phipsirelation}.
\noindent For $k-2s+2 \leq l \leq 2n$:
\begin{align*}
\partial_l\bar{y}_{k,s} &=\left\{\begin{array}{ll}\frac{k-2s+1}{k-4s}\alpha_{s,k,l+n}=0, &k-2s+2\leq l \leq n\\ 
-\frac{k-2s+1}{k-4s}\alpha_{s,k,l-n}=0, & n+1\leq l \leq 2n \end{array}\right.,
\end{align*}
by Lemma \ref{alphalma}.  
For $\bar{z}_{k,s}$ we first consider the case $k-2s=1$; we have:
\begin{align*}
\partial_l\bar{z}_{k,s}&=\sum_{i}[x_{i},dx_{i+n}\wedge dx_{l+n}\omega^s]- [x_{i+n},dx_{i}\wedge dx_{l+n}\omega^s] -[x_{l+n},\omega^{s+1}]\\
&=\frac12\left(\sum_{i}[x_{i},[x_{i+n}, x_{l+n}\omega^s]]- [x_{i+n},[x_i, x_{l+n}\omega^s]]\right) -[x_{l+n},\omega^{s+1}]\\
&=[\omega, x_{l+n}\omega^s] + [\omega^{s+1},x_{l+n}] = (s+2)[\omega,x_{l+n}\omega^s]= (s+2) \phi_s(1\ot a_1),
\end{align*}
by the Jacobi identity.

For $k-2s>1$ we have, by direct computation:
{\footnotesize
\begin{align*}
\partial_l\bar{z}_{k,s} &= \left\{\begin{array}{ll}
\phi_s\left(1\ot p_{l,k-2s-1}\omega - (n-(k-2s)+2) 1\ot a_{k-2s-1}dx_{l+n}\right),& 1\leq l \leq k-2s-1\\ 
\phi_s\left((n-(k-2s)+2)1\ot a_{k-2s-1}dx_{l+n}\right), & k-2s\leq l \leq n,\\
0, & 1\leq l-n\leq k-2s-1\\
\phi_s(-(n-(k-2s)+2)1\ot a_{k-2s-1}dx_{l-n}),& k-2s\leq l-n\leq n\end{array}\right.
\end{align*}
}
The identities in Claim (2) involving $\bar{y}_{k,s}$ and $\bar{z}_{k,s}$ can now be read off directly, as in Proposition \ref{partialprop}, recalling that $\phi_s$ is a morphism of $\mathfrak{sp}_{2n}$-modules. 

Finally, for (3), we begin by noting that, inside $\mathcal{F}_{(2,1^{k-2s})}$ for $k$ odd, the space of $\mathfrak{sp}_{2n}$-highest-weight vectors of weight $\rho_m$ which are killed by all $\partial_i$'s is zero-dimensional if $m$ is even, and one-dimensional if $m$ is odd.  This follows from the fact that the common kernel of all $\partial_i$'s is the generating $\mathfrak{gl}_{2n}$-module $(2,1^k)$
\end{proof}

We now show:

\begin{prop}
The vector $\bar{x}_{k,s}$ is non-zero in $B_3(A_{2n})$.
\end{prop}
\begin{proof}
To show $\bar{x}_{k,s}$ is nonzero, we compute its image under $\theta$ as in Lemma \ref{notdcvgenerator}. As before, we cannot use differential forms, so we use
$$a_{k-2s}^*= 2^{-(k-2s-1)/2}x_1[x_2,x_3[\ldots [x_{k-2s-1},x_{k-2s}]]].$$ Notice that $d$ of the image of $a_{k-2s}^*$ is $a_{k-2s}$. Recall that $\theta(x_1)=ez_0+fz_1$, $\phi(x_i)=z_i$ for $i\ge 2$. 
 
%We first calculate for $v$ of odd degree:
%$$\phi(W(v))=0 $$
%and for $v$ of even degree:
%$$\phi(W(v))=\phi(w)\phi(v) $$
%It follows that for $s\ge 0$,  
%$$\phi(W^s(1))=\phi(\omega)^s=s(ez_0z_{1+n}+fz_1z_{1+n})\left(\sum_{i\ge 2} z_iz_{i+n}\right)^{s-1}+\left(\sum_{i\ge 2} z_iz_{i+n}\right)^{s}$$

%Notice the coefficient of $\left(\sum_{i\ge 2} z_iz_{i+n}\right)^{s-1}$ includes $s$, so this still makes sense when $s=0$. 

We compute:
$$\theta(a_{k-2s})=(ez_0+fz_1)z_2\ldots z_{k-2s}$$
Thus we find:
$$\theta(\bar{x}_{k,s})=[e,f]z_0\ldots z_{k-2s}z_{1+n}\left(\sum_{i\ge 2} z_iz_{i+n}\right)^s$$ 
so $\bar{x}_{k,s}$ is non-zero in $B_3(B)$, and thus in $B_3(A_{2n})$.
\end{proof}

\begin{cor} The vectors $\bar{y}_{k,s}$, $\bar{z}_{k,s}$ are non-zero in $B_3(A_{2n})$.
\end{cor}
\begin{proof}
We have shown that $\bar{x}_{k,s}$ is non-zero in $B_3(A_{2n})$, and it lies in the orbit of both $\bar{y}_{k,s}$ and $\bar{z}_{k,s}$.
\end{proof}

By Lemma \ref{actionlemma}, $B_3(A'_{2n})$ is a quotient of $B_3(A_{2n})$ by the subspace $L_3\cap\langle \omega \rangle$.  Clearly $\bar{y}_{k,s}$, and thus $\bar{x}_{k,s}$, is divisible by $\omega$ for all $s\geq 0$, while $v_{k,s}$ and $\bar{z}_{k,s}$ are divisible by $\omega$ for all $s\geq 1$.  We thus have a surjection
$$\pi:\underset{k \textrm{ odd}}{\bigoplus_{k=1,}^{n-1}}\left(\mathcal{F}_k/Y_k \oplus \mathcal{F}_{(2,1^{k})}\right)\twoheadrightarrow B_3(A'_{2n}).$$

We have computed directly in MAGMA that $v_{k,0}$ and $z_{k,0}$ are non-zero in $B_3(A'_{2n})$ when $2n=4,6$.  Based on this, we conjecture that $\pi$ is an isomorphism for all $n$:

\begin{conj} The vectors $v_{k,0}$ and $z_{k,0}$ are non-zero in $B_3(A'_{2n})$, and thus the $H_{2n}$-module composition factors for $B_3(A'_{2n})$ are:
\begin{align*}
\mathcal{F}_{(2,1^k)},\,\mathcal{F}_{(1^k)}/T_k,\,Z_k/X_k \hspace{.5 in}& \textrm{for $k$ odd, } 1\le k\le n-1.%\\
%\mathcal{F}_{(1^k)}/T_k \hspace{.5 in} & 1\le k\le n-1, \mbox{  k odd} \\
%Z_k/X_k\hspace{.5 in} & 1\le k\le n-1, \mbox{  k odd} \\
\end{align*}
\end{conj}

\end{document}